\documentclass[reqno,11pt]{amsart}
\usepackage{amsfonts}
\usepackage{a4wide}
\usepackage{color}
\usepackage{mathrsfs}
\usepackage{mathtools}
\usepackage{amsmath}
\usepackage{amssymb}
\usepackage{bbm}
\usepackage{esint}
\usepackage{nicefrac}
\numberwithin{equation}{section}

\usepackage[latin1]{inputenc}
\input xy
\xyoption{all}
\usepackage{enumerate}

\DeclareMathOperator{\BV}{BV}
\DeclareMathOperator{\BLV}{BLV}
\DeclareMathOperator{\Var}{V}
\DeclareMathOperator{\LVar}{LV}
\newcommand{\N}{\mathbb{N}}
\newcommand{\R}{\mathbb{R}}
\newcommand{\E}{\mathbb{E}}
\DeclareMathOperator{\T}{T}
\newcommand{\Laakso}{\mathfrak{L}}
\DeclareMathOperator{\dist}{d}
\DeclareMathOperator{\Dist}{dist}

\newcommand{\set}[2]{\{{#1}\mid{#2}\}}
\newcommand{\Set}[2]{\Big\{{#1}\,\Big|\;{#2}\Big\}}
\newcommand{\Sset}[2]{\Bigg\{{#1}\,\Big|\;{#2}\Bigg\}}

\newcommand{\Leqref}[1]{\stackrel{\scriptscriptstyle{\eqref{#1}}}{\leq}}

\newcommand{\eq}[1]{\begin{equation}{#1}\end{equation}}
\newcommand{\mlt}[1]{\begin{multline}{#1}\end{multline}}

\DeclareMathOperator{\LCJ}{LCJ}
\newcommand{\LCJs}{\mathcal{LCJ}}

\newcommand{\FF}{\mathcal{F}}
\newcommand{\AF}{\mathcal{AF}}

\newcommand{\var}{\Var}
\newcommand{\ELL}{\LCJ}
\newcommand{\ELLs}{\ELL_{\seq}}
\newcommand{\eps}{\varepsilon}
\DeclareMathOperator{\seq}{seq}

\DeclareMathOperator{\diam}{diam}
\DeclareMathOperator{\Metr}{\mathcal{X}}
\DeclareMathOperator{\Netr}{\mathcal{Y}}
\DeclareMathOperator{\LIP}{LIP}
\newcommand{\TT}{\mathcal{T}}
\newcommand{\Dyad}{\mathbb{D}}

\renewcommand{\leq}{\leqslant}
\renewcommand{\geq}{\geqslant}

\newcommand{\Leaves}{\mathcal{L}}

\newtheorem{theorem}{Theorem}[section]
\newtheorem{Cor}[theorem]{Corollary}
\newtheorem{Th}[theorem]{Theorem}
\newtheorem{Lm}[theorem]{Lemma}
\newtheorem{Prop}[theorem]{Proposition}
\newtheorem{Def}[theorem]{Definition}
\newtheorem{Example}[theorem]{Example}
\newtheorem{Rem}[theorem]{Remark}

\title{Catching jumps of metric-valued mappings with Lipschitz functions}
\author{Dmitriy Stolyarov}
\thanks{The work of D. Stolyarov is supported by the Ministry of Science and Higher Education of the Russian
Federation (agreement no. 075--15--2025--343)}
\address{St. Petersburg State University, Department of Mathematics and Computer Sciences, 7/9 Universitetskaya emb., St. Petersburg, 199034 Russia}
\email{d.m.stolyarov@spbu.ru}
\author{Alexander Tyulenev}
\address{Steklov Mathematical Institute of Russian academy of Sciences}
\email{tyulenev-math@yandex.ru,tyulenev@mi-ras.ru}
\begin{document}
\date{\today}
\allowdisplaybreaks

\begin{abstract}
It follows from recent results of V. Bakhtin, R. Oleinik, and the second named author
that, given a metric space $\Metr$,
a continuous map $\gamma\colon [a,b] \to \Metr$ is a map of bounded variation if and only if $f \circ \gamma$ is a function of bounded variation for every Lipschitz function $f\colon\Metr \to \mathbb{R}$.
In this note, we show that the continuity assumption is of crucial importance: for many interesting examples of metric spaces there are no analogs of that characterization without the continuity assumption on $\gamma$. The interesting examples are:~$\ell_2$, infinite metric trees, and Laakso-type spaces. However, for ultrametric spaces the said characterization holds without any continuity assumptions.
\end{abstract}
\maketitle

\section{Introduction}
In the seminal paper \cite{Ambrosio1990} L. Ambrosio introduced the theory of mappings of bounded variation attaining values in metric spaces. The approach relied upon post-compositions with Lipschitz functions. Given a nonempty open set $\Omega \subset \mathbb{R}^{d}$ and a nonempty separable
locally compact metric space $\Metr=(\Metr,\rho)$, a Borel map $\gamma\colon \Omega \to \Metr$
is said to be a \textit{map of bounded variation}, provided there is a finite Borel measure $\mu_{\gamma}$ on~$\Omega$ such that for
each Lipschitz function $f\colon\Metr \to \R$ the composition $f \circ \gamma$ is a scalar-valued function of bounded variation and
\eq{\label{eqq.bounded_variation}
|D(f \circ \gamma)|(B) \le L_{f}\mu_{\gamma}(B) \quad \hbox{for each Borel set} \quad B \subset \Omega,
}
where~$D(f \circ \gamma)$ is the distributional gradient of~$f\circ \gamma$, $L_{f}$ is the global Lipschitz constant of~$f$ (defined in~\eqref{LipschitzConstantDefinition} below), and the absolute value of a signed measure $D(f \circ \gamma)$ denotes its total variation measure.

Ambrosio's approach is consistent with the classical one in the one-dimensional case $\Omega=(a,b)$, see \cite{Ambrosio1990} for details.
Since in this note we will deal with the mappings of intervals only, it
will be convenient to recall the classical definition of `one-dimensional mappings' of bounded variation going back to the 19-th century. 

Let $a$ and~$b$ be real numbers such that~$a < b$. Let $\mathfrak{T}$ denote the set of all finite ordered tuples $\T:=\{t_{i}\}_{i=0}^{N} \subset [a,b]$ such that
$a=t_{0} <  t_1 < t_2< \ldots < t_{N}=b$.
Given a metric space~$\Metr=(\Metr,\rho)$, we say that $\gamma\colon [a,b] \to \Metr$ is a \textit{map of bounded variation} and
write $\gamma \in \BV([a,b],\Metr)$, if
\eq{\label{eqq.BV}
\Var_{\gamma}:=\sup\limits_{\T  \in \mathfrak{T}}\sum\limits_{i=1}^{N}\rho(\gamma(t_{i-1}),\gamma(t_{i}))
}
is finite. Similarly, $\gamma\colon [a,b] \to \Metr$ is a \textit{map of bounded Lipschitz variation},  $\gamma \in \BLV([a,b],\Metr)$ provided
\eq{\label{eqq.BLV}
\LVar_{\gamma}:=\sup \Set{\Var_{f \circ \gamma}}{f\colon \Metr \to \R,\ L_f \leq 1}
}
is finite. Clearly,~$\LVar_{\gamma} \leq \Var_\gamma$. One may wonder whether a reasonable reverse estimation is possible. In the case~$\Metr = \R^d$, one has~$\Var_{\gamma} \leq d \LVar_{\gamma}$ by the choice of coordinate projections for~$f$ and a better estimate~$\Var_{\gamma} \lesssim \sqrt{d} \LVar_{\gamma}$ by the choice of all linear functions for~$f$; see Subsection~\ref{s24} below for the explanation. The notation~$A \lesssim B$ means there exists a constant~$C > 0$ such that~$A \leq CB$ and this constant is independent from a parameter, usually this independence is clear from the context. For example, in the estimate above, the bound is dimension-free.

\begin{Prop}[Simplification of Theorem~$1.2$ in~\cite{OT2025}]
\label{Prop.Oleinik_Tyulenev}
Let $\Metr=(\Metr,\rho)$ be a metric space. Then\textup,
\eq{\label{eqq.catches_variation}
\Var_\gamma = \LVar_\gamma \quad \hbox{for each} \quad \gamma \in C([a,b],\Metr).
}
\end{Prop}
Independently and almost simultaneously a slightly weaker version of Proposition~\ref{Prop.Oleinik_Tyulenev} was proved by V.~Bakhtin~\cite{Bakhtin2025} by completely different methods.
More precisely, it was shown in \cite{Bakhtin2025} that $\gamma \in \BLV([a,b],\Metr) \cap C([a,b],\Metr)$ if and only if $\gamma \in \BV([a,b],\Metr) \cap C([a,b],\Metr)$.
The paper~\cite{OT2025} provided characterizations of~$p$-absolutely continuous~$\Metr$-valued curves,~$p \in [1,\infty)$, via Lipschitz post-compositions as well, and focused on geometric measure theoretical properties of such type mappings.

The aim of the present paper is to shed some light on the case where~$\gamma$ is discontinuous. We disregard any Sobolev or geometric measure theory context because the problem about~$\BV$ and~$\BLV$ discontinuous mappings is already rich enough.  As we will see, the dependence on the geometric properties of~$\Metr$ is much stronger.

Before we go into details of the discontinuous case, 
we briefly and informally recall the driving strategy of the proof of Proposition \ref{Prop.Oleinik_Tyulenev}.
Given a map $\gamma\colon [a,b] \to \Metr$, if $f \circ \gamma \in \BV([a,b],\mathbb{R})$ for every Lipschitz function $f\colon\Metr \to \mathbb{R}$, then
one can show that the $1$-Hausdorff
measure of the image $\Gamma:=\gamma([a,b])$ is finite. Then, according to the classical results (see~\cite{Ambrosio1990}) the set $\Gamma$ with the
induced metric appears to be a strongly rectifiable metric space.
Hence, one can use the full strength of the classical geometric measure theory machinery (for example, area-type formulas) and make some sort of reduction to `locally Euclidean settings'.
Furthermore, it is sufficient to care about infinitesimal behavior of Lipschitz functions only (i.e. about Jacobians of Lipschitz functions with respect to
$\Gamma$). As we will see, all this heavy machinery is not applicable to the discontinuous case.

\begin{Def}
\label{Def.catches_jumps}
Given $\Metr=(\Metr,\rho)$\textup,
we say that on~$\Metr$ Lipschitz functions catch jumps\textup, provided for any map~$\gamma \colon [a,b]\to \R$ with~$\LVar_\gamma < \infty$\textup, the total variation of~$\gamma$ is also finite.
The class of such metric spaces is denoted by~$\LCJs$.
\end{Def}

It occurs that many interesting metric spaces do not fall in the class $\LCJs$.
The simplest examples include infinite metric trees, infinite dimensional Banach spaces, and Laakso-type spaces. To develop an approach to the phenomenon, we consider a quantification of the~$\LCJs$-property:
\eq{\label{EllOfM}
\LCJ(\Metr) = \inf \Set{\frac{\LVar_\gamma}{\var_\gamma}}{\gamma\in \BV([a,b],\Metr)}.
}
Clearly~$\LCJ(\Metr) \in [0,1]$, and if~$\LCJ(\Metr) > 0$, then Lipschitz functions catch jumps on~$\Metr$, i.e. $\Metr \in \LCJs$.
It is easy to see that~$\LCJ(\R) = 1$,
and~$\LCJ(\R^d)\gtrsim d^{-\frac12}$, see Subsection~\ref{s24} for an explanation of the second inequality.

Whenever unspecified, we equip~$\R^d$ with the standard Euclidean norm. The first \emph{main result} of this note reads as follows.
\begin{Th}\label{Th1}
For each~$d > 1$\textup,
\eq{
\LCJ(\R^d) \lesssim d^{-\frac12}(\log d)^{\frac12}.
}
\end{Th}
\begin{Cor}\label{Cor1}
For any infinite dimensional Banach space~$\Metr$\textup,~$\LCJ(\Metr) = 0$.
\end{Cor}

The corollary above might have led to a suggestion that the doubling property of the space might be related to the~$\LCJs$-property.
Unfortunately, this is not the case. Our \textit{second main result} below shows that there are non doubling metric spaces where Lipschitz functions catch jumps.
We refer the reader to Chapter~$15$ of~\cite{DavidSemmes1997} for information about uniformly disconnected metric spaces and their relationship
with ultrametric spaces; we provide a digest in Section~\ref{S2}.
\begin{Th}\label{UltrametricTheorem}
Let~$\Metr$ be a uniformly disconnected metric space. Then\textup,~$\LCJ(\Metr) > 0$.
\end{Th}
As we will see from the proof of Theorem~\ref{UltrametricTheorem}, this bound is uniform with respect to the parameters in the definition of uniform disconnectivity in Subsection~\ref{s21}. The proof of Theorem~\ref{UltrametricTheorem} is relatively simple once we accept the proper idea to use random Lipschitz functions.

On the other hand, there are doubling metric spaces where Lipschitz functions do not catch jumps.
The Laakso-type spaces were introduced in \cite{Laakso2000}. They were the
first examples of Ahlfors $Q$-regular metric spaces,~$Q \geq 1$ being arbitrary, satisfying Poincare-type inequalities. Those spaces
are doubling, but do not admit a bi-Lipschitz embedding into $\R^{d}$, $d \in \N$. We present our \textit{third main result}.
\begin{Th}
\label{LaaksoTheorem}
There exists a Laakso space~$\Laakso$ that is not in the~$\LCJs$-class.
\end{Th}
Finally, we consider the simplest metric trees. This example is quite interesting because it allows one to demonstrate both phenomena
described above.
\begin{Th}\label{TreeTheorem}
Let~$\TT$ be the dyadic tree of depth~$N$ equipped with the standard graph metric. Then\textup,
\eq{\label{WholeTreeInequality}
\LCJ(\TT) \lesssim (\log N)^{-\frac12}.
}
However\textup, if~$\Leaves$ denotes the set of its leaves with the metric induced from~$\TT$, then,
\eq{\label{LeavesInequality}
\LCJ(\Leaves) \gtrsim 1.
}
\end{Th}
\begin{figure}[h!]
\centerline{
\includegraphics[height=6cm]{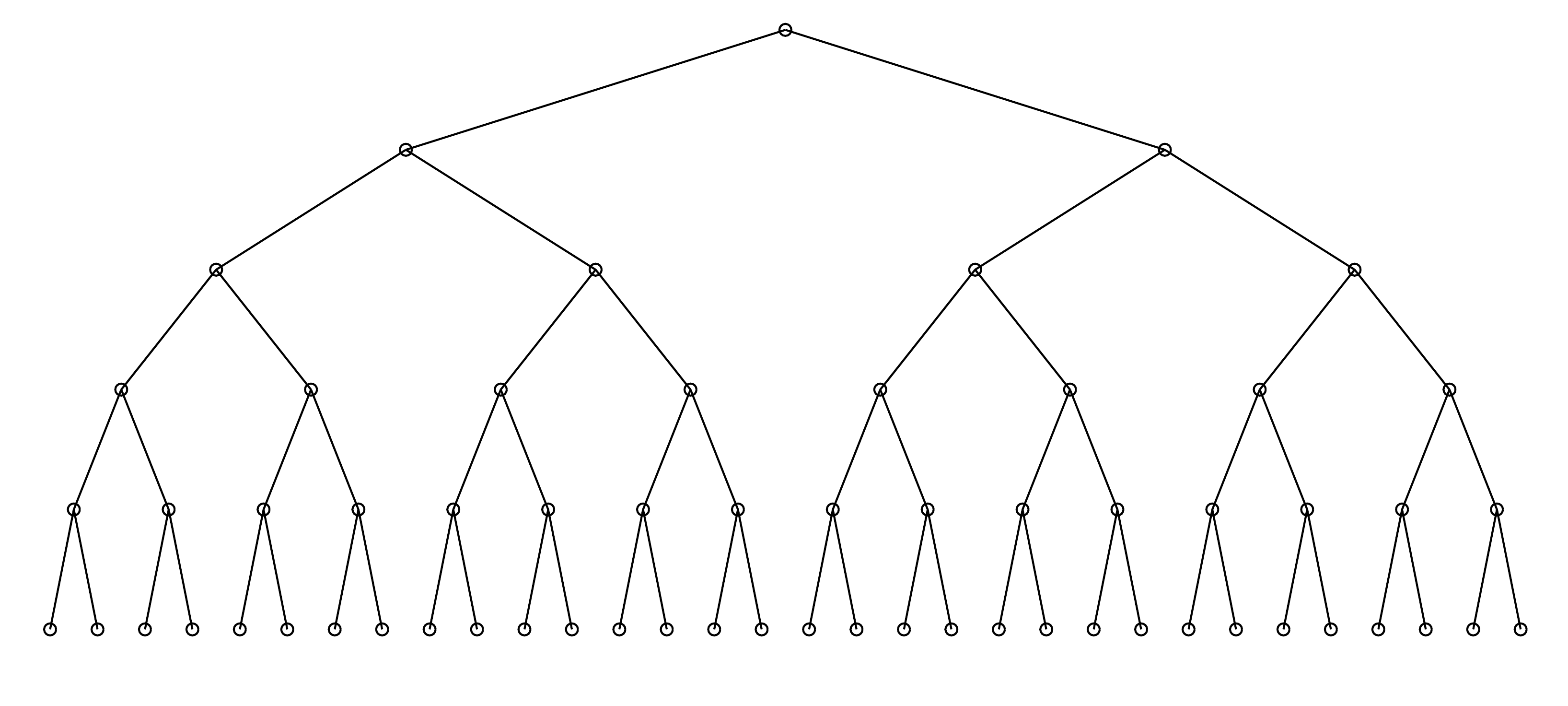}}
\caption{Dyadic tree for~$N=5$.}
\label{DyadicTree}
\end{figure}

The second assertion in Theorem~\ref{TreeTheorem} follows from Theorem~\ref{UltrametricTheorem} since~$\Leaves$ may be equipped with an ultrametric equivalent to the standard graph metric: The distance between two leaves equals the level number of their latest common ancestor.

The proof of Theorem \ref{LaaksoTheorem} and justification of~\eqref{WholeTreeInequality} 
are slightly more involved. Surprisingly,
the crucial ideas for the proofs are the same and rely upon a special martingale and an orthogonality argument for it.
In Section~\ref{S4}, we present this common
idea in the setting of  a general metric space and then consider the two cases of interest.

A comment on the choice of examples. Theorem~\ref{Th1} shows that some infinite dimensional spaces may violate the~$\LCJs$-property. We treat dimension from the metric space standpoint: Non-doubling spaces are infinite dimensional. A more sophisticated Theorem~\ref{LaaksoTheorem} shows that infinite dimension is not necessarily the source of that violation. The Laakso space is finite dimensional, however, it is `highly curved' as it does not have uniqueness of local geodesics. The example of the metric tree shows that the curvature has also nothing to do with the~$\LCJs$-property since trees are non-positively curved and~$0$-hyperbolic. It is interesting whether any metrically doubling subset of a Hilbert space possesses the~$\LCJs$-property. The problem of bi-Lipschitz embedding of such metric spaces into a Euclidean space goes back to~\cite{LangPlaut2001} and has attracted much attention. As we know from Theorem~\ref{UltrametricTheorem},~$\LCJs$ is weaker than embedding into a Euclidean space.

The contents of the rest of the paper are: Section~\ref{S2} is preliminary, it provides notation and special background; Section~\ref{S3} contains reasoning related to the concentration of measure phenomenon and provides the proofs of Theorem~\ref{Th1} and Corollary~\ref{Cor1}; in Section~\ref{S4}, we prove Theorem~\ref{LaaksoTheorem} and the first part of Theorem~\ref{TreeTheorem}, the reasonings are based upon orthogonality techniques; the final Section~\ref{S5} is devoted to the proof of Theorem~\ref{UltrametricTheorem} by construction of a certain random Lipschitz function. 

We wish to thank Roman Karasev for exposition advice and Mikhail Lifshits for advice on ultrametric spaces related to the formulation of Theorem~\ref{UltrametricTheorem}.
\noindent

\section{Preliminaries and notation}\label{S2}

\subsection{Metric spaces}\label{s21}
By a ball~$B_r(x)$ in $\Metr=(\Metr,\rho)$ we always mean the closed ball with a given center~$x\in \Metr$ and radius~$r > 0$. 
For each $x \in \Metr$ and $r > 0$,
we denote by $N_{r}(x) \in \N \cup \{+\infty\}$ the minimal number of balls of radii $r/2$ needed to cover $B_{r}(x)$. We say that $\Metr$ is \textit{metrically doubling} if
\eq{
\sup\limits_{x \in \Metr} \sup\limits_{r > 0}N_{r}(x) < +\infty.
}

Given nonempty metric spaces $\Metr=(\Metr,\rho)$ and $\Netr=(\Netr,\dist)$, define the global Lipschitz constant of a map~$F\colon \Metr \to \Netr$ as
\eq{\label{LipschitzConstantDefinition}
L_F = \sup\limits_{x_1\ne x_2} \frac{\dist(F(x_1),F(x_2))}{\rho(x_1,x_2)}.
}
We say~$F$ is~$L$-Lipschitz, provided~$L_F \leq L$. We denote the set of all $L$-Lipschitz maps from $\Metr$ to $\Netr$ by $\operatorname{LIP}_{L}(\Metr,\Netr)$.

Following \cite{DavidSemmes1997} we say that a metric spaces $\Metr=(\Metr,\rho)$ is \textit{uniformly disconnected} if there is a constant $c > 0$ such that for
each $x \in \Metr$ and $r > 0$ we can find a closed set $A$ such that $B_{cr}(x) \subset A \subset B_{r}(x)$ and
$\Dist(A,\Metr \setminus A) \geq cr$. We also recall that a metric $\rho'$ on  a set $\Metr$ is called an \textit{ultrametric} if
$\rho'(x,z) \le \max\{\rho'(x,y),\rho'(y,z)\}$ for all triples $x,y,z \in \Metr$. An ultrametric space is uniformly disconnected.
\begin{Prop}[Proposition~$15.7$ in~\cite{DavidSemmes1997}]
\label{Prop.ultrametric}
Given a uniformly disconnected metric space~$\Metr=(\Metr,\rho)$\textup, there is an ultrametric $\rho'$ on $\Metr$ and a constant $C > 0$ such that
\eq{
\frac{1}{C}\rho'(x,y) \le \rho(x,y) \le C\rho'(x,y) \quad \hbox{for all} \quad (x,y) \in \Metr \times \Metr.
}
\end{Prop}

 \subsection{Martingales}
As we have said, we will use some portions of elementary probability theory. While we assume the reader is familiar with the notion of i.i.d. random variables, the provision of some martingale toolkit seems desirable; we also refer broadly to Chapter~VII of~\cite{Shiryaev2019} for the general martingale theory. We fix a probability space
$(\Omega,\FF,\operatorname{P})$. Let~$\FF_0 \subset \FF_1 \subset \FF_2 \subset\ldots$ be an increasing, possibly finite, sequence of~$\sigma$-subfields of~$\FF$.  Assume for simplicity all these~$\sigma$-fields are finite. In other words, we have a sequence of partitions of~$\Omega$ into finite number of parts that increases in the sense that the partitions deliver more information as~$n$ increases. Recall that the sets of the partition are called atoms, they naturally correspond to indivisible elements of the~$\sigma$-field. Denote the set of all atoms in~$\FF_n$ by~$\AF_n$. The sequence~$M_0, M_1,\ldots$ of random variables is called a martingale, provided it meets two requirements: Each~$M_n$ is~$\FF_n$-measurable and~$\E(M_{n+1}\mid \FF_n) = M_n$. We will use the notation
\eq{
dM_{n} = M_n - M_{n-1},\qquad n \in \N,
}
to denote the martingale difference sequence. If~$\omega \in \AF_n$, then the function~$dM_{n+1}|_{\omega}$ attains finitely many values and has mean zero. One may verify that~$d M_{n}$ and~$d M_m$ are orthogonal in the sense that
\eq{\label{Orthogonal}
\E (dM_{n}\cdot d M_m) = 0, \qquad m\ne n.
}
Finally, we say that a martingale~$M$ is bounded, provided~$\sup_n \|M_n\|_{L_\infty} < \infty$; we also set
\eq{
\|M\|_{L_p}  = \sup_n \|M_n\|_{L_p},\qquad 1 \leq p \leq \infty.
}
The sequence~$\|M_n\|_{L_p}$ is non-decreasing for any~$p \geq 1$; this follows from Jensen's inequality.

\begin{Example}\label{OneDimensionalExample}
Let~$f\colon [0,1]\to \R$ be a Lipschitz function. Let~$\Dyad_n$ be the~$\sigma$-field generated by the~$n$-th generation dyadic subintervals of~$[0,1]$. Consider the piecewise constant functions
\eq{\label{MartingaleExampleFormula}
M_n(t) = 2^n\Big(f(2^{-n} k) - f(2^{-n}(k-1))\Big),\quad t\in \Big[2^{-n}(k-1), 2^{-n}k\Big), \quad k = 1,2,\ldots, 2^n.
}
These functions form a martingale~$M$ adapted to~$\Dyad_n$\textup, which also satisfies the bound
\eq{
\|M\|_{L_\infty}\leq L_f.
}
We also note that
\eq{
dM_{n+1}(t) = \begin{cases}
2^n\big(2f(B) - f(C) - f(A)\big), \quad &t\in \big[A, B\big);\\
2^n\big(f(C) + f(A) - 2f(B)\big), \quad &t\in \big[B, C\big),
\end{cases}
}
where~$A = 2^{-n}(k-1)$\textup,~$B = 2^{-n-1}(2k-1)$\textup, and~$C = 2^{-n}k$. Thus\textup,
\eq{\label{eq210}
\E |d M_{n+1}|\chi_{[2^{-n}(k-1), 2^{-n}k)} = \Big|f(C) + f(A) - 2f(B)\Big|.
}
\end{Example}
\begin{figure}[h!]
\centerline{
\includegraphics[height=3.3cm]{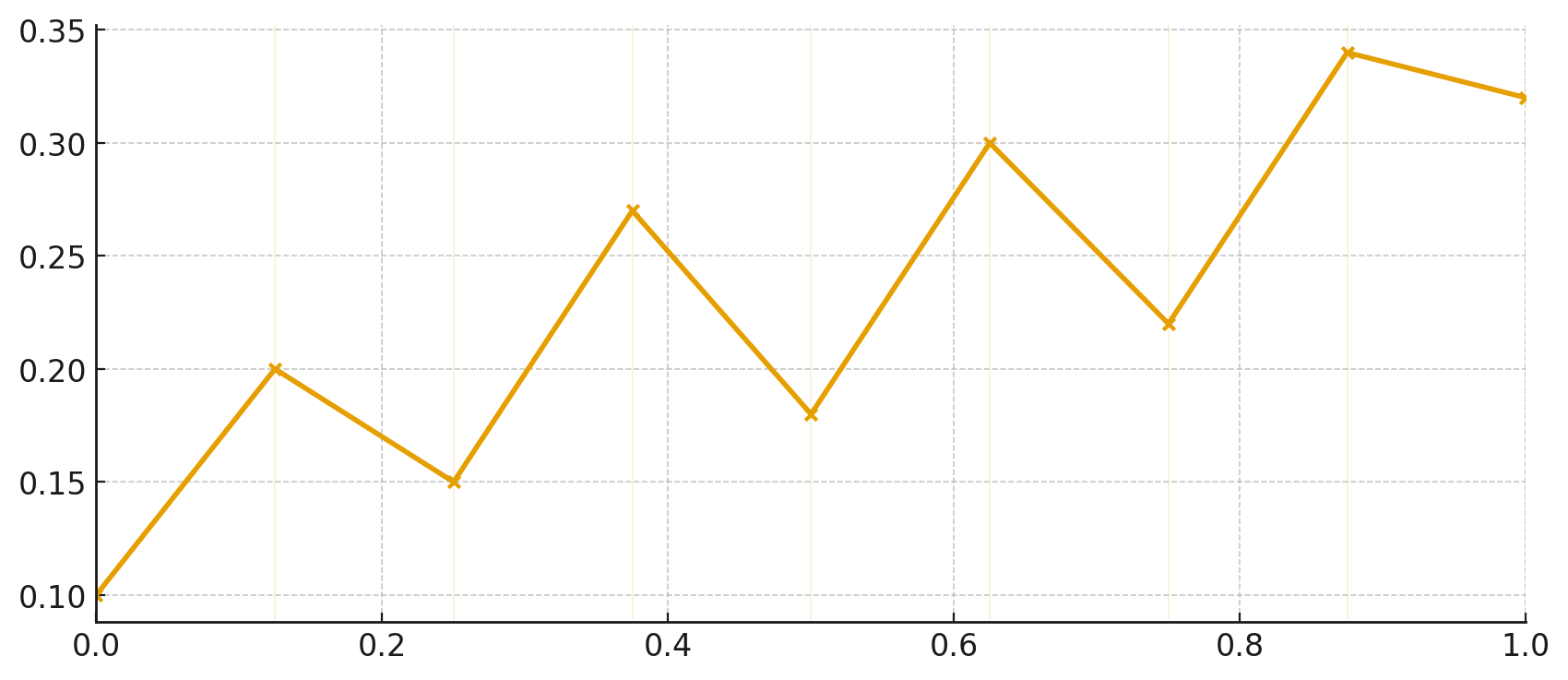}
\includegraphics[height=3.3cm]{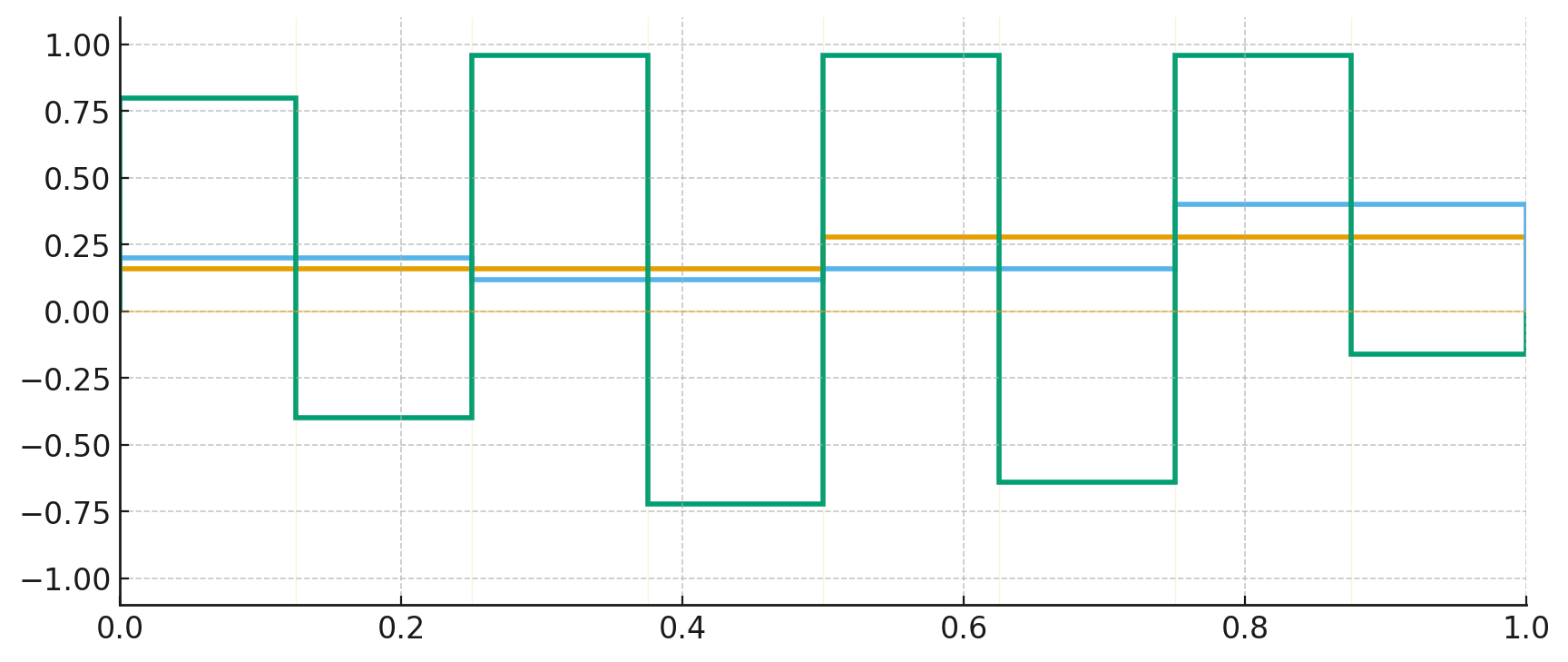}
}
\caption{A Lipschitz function~$f$ and functions~$M_1$ (orange),~$M_2$ (blue), and~$M_3$ (green) constructed from it via~\eqref{MartingaleExampleFormula}.}
\end{figure}

We conclude the martingale toolkit section with an orthogonality argument that plays the pivotal role in the proofs of Theorem~\ref{LaaksoTheorem} and the bound~\eqref{WholeTreeInequality}.
\begin{Lm}
\label{MartingaleOrthogonality}
Let $M$ be a martingale. Then\textup, for each $N \in \mathbb{N}$\textup,
\eq{
\label{eqq.key_martingal_ineq}
\sum\limits_{n=1}^N\|dM_n\|_{L_1}\leq \sqrt{N}\|M\|_{L_\infty}.
}
\end{Lm}

\begin{proof}
We use the aforementioned orthogonality in combination with the Cauchy--Schwarz inequality:
\eq{
\begin{split}
&\sum\limits_{n=1}^N\|dM_n\|_{L_1} \leq \sum\limits_{n=1}^N\|dM_n\|_{L_2} \\ 
&\leq \sqrt{N} \Big(\sum\limits_{n=1}^N\|dM_{n}\|_{L_2}^2\Big)^\frac12 
 \Leqref{Orthogonal}  \sqrt{N}\|M\|_{L_2} \leq  \sqrt{N}\|M\|_{L_\infty}.
\end{split}
}
\end{proof}

\subsection{Laakso space}\label{sLaakso}
Here we briefly recall the simplest example of the so-called Laakso-type spaces. The literature concerning those spaces is vast. We refer the reader to the pioneering paper \cite{Laakso2000} and the recent survey in 
\cite{Capolli2024} (see also references therein). We will be working with the particular case of a Laakso-type space and will follow the exposition in~\cite{LangPlaut2001}. The metric space will be constructed as a limit of a sequence of finite metric spaces~$X_i$ such that~$X_0 \subset X_1 \subset X_2 \subset \ldots$. The space~$X_0$ is the pair of points~$\{0,1\}$ considered as the subset of the unit interval, which is identified with the graph~$G_0$. To construct~$X_1$, we split the unit interval it into four intervals of equal length and replace the two middle intervals with a rhombus as shown on Fig.~\ref{Fig:Laakso}. We equip the obtained graph~$G_1$ with~$1/4$ of the natural graph metric and set~$X_1$ to be its set of vertices. Note that~$X_0 \subset X_1$ in a natural way. Then, we replace each edge of~$G_1$ with a copy of~$G_1$, get a new graph~$G_2$ as shown on Fig.~\ref{Fig:Laakso}, equip it with~$1/16$ of the natural graph metric, and set~$X_2$ to be its set of vertices. We again have~$X_1 \subset X_2$ in a natural way.  Then we replace each edge of~$G_2$ with a copy of~$G_1$, equip the obtained graph~$G_3$ with~$1/64$ of the standard graph metric, and so forth. The Laakso space is the Gromov--Hausdorff limit of the sequence~$\{X_i\}_{i\in \N}$. One may verify that the obtained metric space is doubling. The will also use the  corresponding graphs~$G_i$ in our notation since we will be working not only with the vertices~$X_i$, but with the paths in~$G_i$ as well. We treat~$G_i$ as metric space, and the metric is~$4^{-i}$ of the natural graph metric. 

\begin{figure}[h!]
\centerline{
\includegraphics[height=3cm]{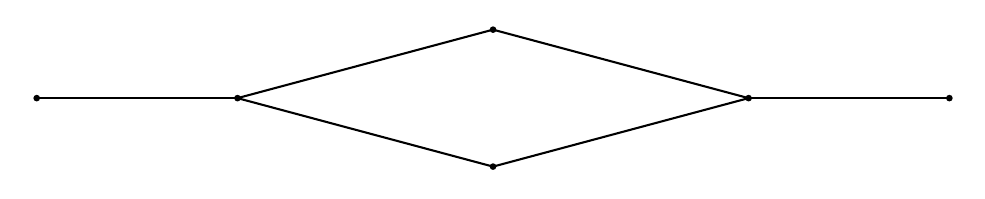}
}
\centerline{
\includegraphics[height=2.9cm]{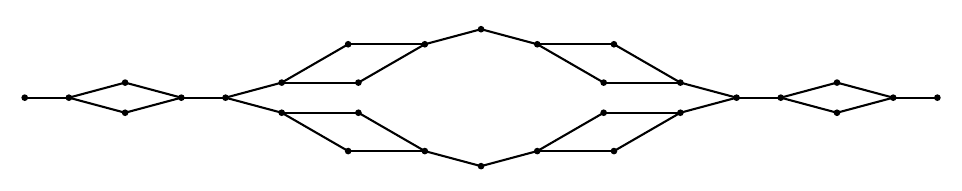}
}
\centerline{
\includegraphics[height=2.8cm]{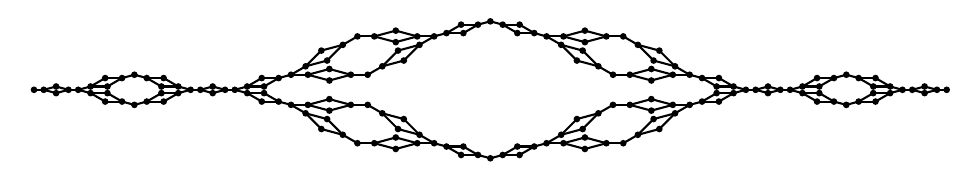}
}
\caption{The graphs~$G_1$,~$G_2$, and~$G_3$.}
\label{Fig:Laakso}
\end{figure}

\subsection{Catching jumps with linear functions on~$\R^d$}\label{s24}
In this subsection, we will show the bound~$\LCJ(\R^d)\gtrsim d^{-\frac12}$ claimed before. Let~$\pi_h$ be the orthogonal projection onto the line~$\R\cdot h$ spanned by a vector~$h\in \R^d$; we will also write~$\pi_j$ for the said projection in the case~$h=e_j$, the~$j$-th basic vector,~$j=1,2,\ldots, d$. Clearly,~$\pi_h\in \LIP_1(\R^d)$, provided~$\|h\|=1$. Now we will show what can be achieved by substituting~$h$ for~$f$ into~\eqref{eqq.BLV}. First, the inequality 
\eq{
\Var_\gamma \leq \sum\limits_{j=1}^d \Var_{\pi_j\circ\gamma}
}
yields the bound~$\Var_\gamma \leq d \LVar_\gamma$ for any~$\gamma \colon [0,1]\to \R^d$, which implies~$\LCJ(\R^d) \geq d^{-1}$. Note that one cannot obtain a better estimate by relying on coordinate projections~$\pi_j$ only since for the mapping
\eq{
\gamma(t) = \sum\limits_{j=1}^{[td]}e_j + \{td\}e_{[td]},\qquad t\in [0,1],
}
we have~$\Var_{\gamma} = d$ and~$\Var_{\pi_j\circ\gamma} = 1$ for any~$j=1,2,\ldots, d$; here we have used the notation~$[x]$ and~$\{x\}$ for the integer and fractional parts of a number~$x\in\R$.

To prove a better bound, pick~$\{\eps_j\}_{j=1}^d$ be the centered independent Bernoulli random variables and consider~$h = d^{-\frac12}\sum_j \eps_j e_j$ to be a random vector of unit length. Then, for any~$v\in \R^d$, we have
\eq{
\E|\pi_{h}[v]| = \E d^{-\frac12}\Big|\sum\limits_{j=1}^d \eps_j v_j\Big| \asymp d^{-\frac12} \|v\|
}
by Khintchine's inequality. Thus, for any mapping~$\gamma \colon [0,1]\to \R^d$, we have
\eq{
\E\Var_{h\circ\gamma} \asymp d^{-\frac12} \Var_{\gamma},
}
which implies~$\LCJ(\R^d) \gtrsim d^{-\frac12}$. Note that the only Lipschitz functions we have used to establish this bound are the linear ones. Can one obtain better bounds by employing some more sophisticated functions and fill the gap between this bound and the upper estimate obtained in Theorem~\ref{Th1}?

\section{What is given by the concentration of measure?}\label{S3}

\begin{Rem}\label{UnitLength}
For a Banach space~$X$,
\eq{
\ELL(X) = \inf \Set{\LVar_\gamma}{\gamma \colon [a,b] \to X, \var_\gamma = 1}.
}
According to the definition~\eqref{EllOfM}\textup, the proof reduces to the following assertion\textup: For any~$\gamma \colon [a,b]\to X$ there exists a mapping~$\tilde{\gamma} \colon [a,b]\to X$ such that~$\var_{\tilde{\gamma}} = 1$ and~$\LVar_{\tilde{\gamma}} = \LVar_{\gamma}/ \var_{\gamma}$\textup, see~\eqref{EllOfM}. To justify the assertion, pick an arbitrary mapping~$\gamma\colon [a,b] \to X$. Then, for any~$\lambda > 0$\textup, consider the mapping
\eq{
\gamma_\lambda(t) = \lambda \gamma(t),\qquad t\in [a,b];
}
which is a dilation of~$\gamma$. Then\textup,~$\var_{\gamma_\lambda} = \lambda \var_\gamma$ and if we set~$f_\lambda(x) = \lambda f(x/\lambda)$\textup, then~$\var_{f_\lambda \circ \gamma_\lambda} = \lambda \var_{f\circ \gamma}$ and~$f_\lambda$ is~$1$-Lipschitz\textup, provided~$f$ is. The choice~$\lambda = (\var_\gamma)^{-1}$ proves the assertion.
\end{Rem}

 \begin{proof}[Derivation of Corollary~\ref{Cor1} from Theorem~\ref{Th1}]
Let~$X_1,X_2,\ldots, X_n,\ldots$ be linear subspaces of~$X$, each~$X_n$ being~$2$-linearly isomorphic to~$\R^{2^n}$; such spaces exist by the Dvoretsky theorem (see, e.g., Section~$19$ in~\cite{DJT1995}). By Theorem~\ref{Th1} and Remark~\ref{UnitLength}, there exists a mapping~$\gamma_n\colon [0,1] \to X_n$ of unit variation such that
\eq{
\LVar_{\gamma_n} \lesssim \sqrt{n}2^{-n/2}.
}
Construct the concatenation of these mappings
\eq{
\tilde{\gamma}_{n+1}(t) =
\begin{cases}
\tilde{\gamma}_n(n) + \gamma_{n+1}(t) - \gamma_{n+1}(0),\quad &t\in [n,n+1];\\
\tilde{\gamma}_n(t),\quad &t\in [0,n).
\end{cases}
}
In this case,~$\var_{\tilde{\gamma}_n} = n$ and we have
\eq{
\var_{f\circ \tilde{\gamma}_n} \leq \sum\limits_{m \leq n} \sqrt{m} 2^{-\frac{m}{2}}\lesssim 1
}
for any~$1$-Lipschitz function~$f$.
\end{proof}

For the proof of Theorem~\ref{Th1}, we need to modify the quantity~$\ELL(\Metr)$ a little bit. Let~$\bar{x} = \{x_n\}_n$ and~$\bar{y} = \{y_n\}_n$ be two sequences in a metric space~$\Metr$. Set
\eq{
\var(\bar x, \bar y) = \sum\limits_n \rho(x_n,y_n)
}
and define
\eq{
\LVar(\bar x, \bar y) = \sup \Set{\sum\limits_n |f(x_n) - f(y_n)|}{f \colon \Metr \to \R \text{ is a $1$-Lipschitz function}}.
}
Finally, let
\eq{
\ELLs(\Metr) = \inf\Set{\frac{\LVar(\bar x, \bar y)}{\var(\bar x, \bar y)}}{\bar x, \bar y \text{ are finite sequences in } \Metr,\quad \bar{x} \ne \bar{y}}.
}

\begin{Lm}\label{FromCurvesToSequences}
For any metric space~$\Metr$\textup, we have~$\operatorname{LCJ}_{\rm seq}(\Metr) = \operatorname{LCJ}(\Metr)$.
\end{Lm}
\begin{proof}
The proof naturally splits into two parts.

The {\bf first part} is the verification of the inequality~$\ELLs(\Metr)\leq \ELL(\Metr)$. Fix~$\eps > 0$ and pick~$\gamma \colon [0,1]\to \Metr$ such that
\eq{
\frac{\LVar_\gamma}{\Var_\gamma} \leq \ELL(\Metr)+\eps.
}
By the very definition, there exists a finite increasing sequence~$\{t_{i}\}_{i=0}^{N} \subset [0,1]$ such that
\eq{
\Var_\gamma \leq (1+\eps)\sum\limits_{i=1}^N\rho(\gamma(t_{i-1}),\gamma(t_i)).
}
We set~$x_i := t_{i-1}$,~$y_i := t_i$ for~$i=1,2,\ldots, N$ and obtain
\eq{
\begin{aligned}
\Var(\bar x, \bar y)  &\geq (1+\eps)^{-1} \Var_\gamma;\\
\LVar(\bar x, \bar y) &\leq \LVar_\gamma,
\end{aligned}
}
which yields
\eq{
\ELLs(\Metr) \leq \frac{\LVar(\bar x, \bar y)}{\Var(\bar x, \bar y)} \leq (1+\eps) \frac{\LVar_\gamma}{\Var_\gamma}\leq (1+\eps)(\ELL(\Metr) + \eps).
}
Since~$\eps$ is arbitrarily small, this implies the desired inequality.

In the {\bf second part}, we wish to prove the bound~$\ELL(\Metr) \leq \ELLs(\Metr)$: For any pair~$(\bar x, \bar y)$ of finite sequences and a given~$\eps > 0$, we wish to construct a mapping~$\gamma \colon [0,1]\to \Metr$ such that
\eq{\label{eq310}
\frac{\LVar_\gamma}{\var_\gamma} \leq \frac{\LVar(\bar x,\bar y)}{\var(\bar x, \bar y)} + \eps.
}
Pick a large number~$K$. Define the searched-for mapping~$\gamma$ by the following algorithm: it jumps from~$x_1$ to~$y_1$ and back~$2K$ times, then goes from~$x_1$ to~$x_2$; jumps from~$x_2$ to~$y_2$ and back~$2K$ times, then goes from~$x_2$ to~$x_3$; and so forth till~$x_N$. Then,~$\Var_\gamma \geq 2K\var(\bar x, \bar y)$.

Let~$f\in \LIP_1(\Metr)$. Then, the total variation of~$f\circ \gamma$ does not exceed
\eq{
2K \LVar(\bar x, \bar y) + \sum_j \rho(x_j,x_{j+1}).
}

Thus, 
\eq{
\frac{\LVar_\gamma}{\var_\gamma} \leq \frac{\LVar(\bar x, \bar y)}{\var(\bar x,\bar y)} + o(1), \quad \text{as}\quad  K \to \infty,
}
which justifies~\eqref{eq310} by the choice of sufficiently large~$K$.
\end{proof}
\begin{Cor}\label{MeasureRemark}
One may go further and replace a pair of sequences by a finite measure~$\mu$ on~$\Metr\times \Metr$\textup:
\mlt{\label{MeasureRemarkIdentity}
\ELL(\Metr)\\ = \inf \Sset{\sup \Set{\frac{\int_{\Metr\times \Metr} |f(x) - f(y)|\,d\mu(x,y)}{\int_{\Metr\times \Metr} \rho(x,y)\,d\mu(x,y)}}{f \in \LIP_1(\Metr)}}{\mu\text{ measures }\Metr\times \Metr},
}
we assume that~$\mu$ is not concentrated on the diagonal~$\set{(x,x)\in \Metr\times \Metr}{x\in \Metr}$ and that the integral in the denominator is finite. 
\end{Cor}
\begin{proof}
We will denote the quantity on the right hand side of~\eqref{MeasureRemarkIdentity} by~$\mathfrak{S}$ during the proof.

To each pair~$(\bar x, \bar y)$ of finite sequences of points in~$\Metr$, there corresponds a measure~$\mu_{(\bar x,\bar y)} := \sum_{i=1}^N \delta_{(x_i,y_i)}$. This correspondence and Lemma~\ref{FromCurvesToSequences} show that~$\ELL(\Metr)$ is greater or equal to~$\mathfrak{S}$. To prove the reverse bound, we wish to somehow approximate an arbitrary probability measure~$\mu$ on~$\Metr\times \Metr$ by the measures of the type~$\sum_{i=1}^N \delta_{(x_i,y_i)}$. Note that by duplicating points in the sequences~$\bar x$ and~$\bar y$ and scaling by a multiplicative constant, we may enlarge the class of possible approximations by the measures of the type~$\sum_{i=1}^N a_i \delta_{(x_i,y_i)}$,~$a_i \in \mathbb{Q}$.

To this end, fix~$\eps > 0$ and a finite measure~$\mu$ on~$\Metr\times \Metr$ such that~$\int_{\Metr\times\Metr} \rho(x,y)\,d\mu(x,y)$ is finite and positive and, moreover,
\eq{
\frac{\int_{\Metr\times \Metr} |f(x) - f(y)|\,d\mu(x,y)}{\int_{\Metr\times \Metr} \rho(x,y)\,d\mu(x,y)} \leq (1+\eps)\mathfrak{S}.
} 
We split~$\Metr\times \Metr$ into a disjoint union of Borel sets~$\{Q_\alpha\}_{\alpha}$ such that~$\diam Q_\alpha < \delta$, the latter symbol denotes yet another small parameter. If a set~$Q_\alpha$ is not contained in the diagonal, let~$x_\alpha \in Q_\alpha$ be a point not on the diagonal. Construct a measure
\eq{\label{muepsDef}
\mu_\delta = \sum\limits_{\alpha} \mu(Q_\alpha) \delta_{x_\alpha}.
}
In this sum, at most countable number of summands  are non-zero since~$\mu$ is finite; the measure~$\mu_\delta$ is finite as well. Note that
\eq{
\Big|\int\limits_{\Metr\times \Metr} g(x,y)\,d\mu_\delta(x,y) - \int\limits_{\Metr\times \Metr} g(x,y)\,d\mu(x,y)\Big| \leq \delta
}
for any~$g\in \LIP_1(\Metr\times\Metr)$. Thus,
\eq{
\frac{\int_{\Metr\times \Metr} |f(x) - f(y)|\,d\mu_\delta(x,y)}{\int_{\Metr\times \Metr} \rho(x,y)\,d\mu_\delta(x,y)} \leq \frac{\int_{\Metr\times \Metr} |f(x) - f(y)|\,d\mu(x,y) + 2\delta}{\int_{\Metr\times \Metr} \rho(x,y)\,d\mu(x,y) - \delta} 
}
for any~$f\in \LIP_1(\Metr)$. If~$\delta$ is sufficiently small, this yields
\eq{
\frac{\int_{\Metr\times \Metr} |f(x) - f(y)|\,d\mu_\delta(x,y)}{\int_{\Metr\times \Metr} \rho(x,y)\,d\mu_\delta(x,y)} \leq (1+\eps)^2 \mathfrak{S}
}
for some choice of~$f$. The same approximation argument allows to restrict the sum in~\eqref{muepsDef} to a finite set of indices and replace the coefficients~$\mu(Q_\alpha)$ with their rational approximations; let this modified measure be~$\tilde{\mu}$. As we have seen before, such measures are generated by pairs of finite sequences of points in~$\Metr$, and, therefore,
\mlt{
\frac{\LVar(\bar x, \bar y)}{\Var(\bar x, \bar y)}  = \frac{\int_{\Metr\times \Metr} |f(x) - f(y)|\,d\tilde{\mu}(x,y)}{\int_{\Metr\times \Metr} \rho(x,y)\,d\tilde{\mu}(x,y)}\\
 \leq (1+\eps) \frac{\int_{\Metr\times \Metr} |f(x) - f(y)|\,d\mu_\delta(x,y)}{\int_{\Metr\times \Metr} \rho(x,y)\,d\mu_\delta(x,y)} \leq (1+\eps)^3 \mathfrak{S}
}
for some finite sequences~$\bar x$ and~$\bar y$. Since~$\eps$ is arbitrarily small, this yields~$\ELL(\Metr) \leq \mathfrak{S}$.
\end{proof}

Let~$S^{d-1}$ be the unit sphere in~$\R^d$. Let~$\lambda_{d-1}$ be the Lebesgue measure on~$S^{d-1}$. It is convenient to normalize~$\lambda_{d-1}$ to the probability measure and treat~$(S^{d-1},\lambda_{d-1})$ as a probability space. We will be using the following classical theorem, see Corollary~$19.31$ in~\cite{DJT1995}.
\begin{Th}[Levy's concentration inequality]
Assume~$d \geq 4$. Let~$f\colon S^{d-1}\to \R$ be a~$1$-Lipschitz function\textup, consider it as a random variable. Let~$c$ be its median. Then\textup,
\eq{\label{ConcentrationOfMeasure}
\lambda_{d-1}\Big(\Set{x\in S^{d-1}}{|f(x) - c| > \eps}\Big) \leq 6e^{-\frac{d\eps^2}{8}}.
}
\end{Th}
\begin{proof}[Proof of Theorem~\ref{Th1}]
We wish to prove a slightly stronger statement~$\ELL(S^{d-1}) \lesssim d^{-\frac12}(\log d)^{\frac12}$. We will be using the formula for~$\ELL$ given in Corollary~\ref{MeasureRemark}.  The measure~$\mu$ on~$S^{d-1}\times S^{d-1}$ is given implicitly by the formula
\eq{
\int\limits_{S^{d-1}\times S^{d-1}}\!\!\!\!\! g\,d\mu = \int\limits_{S^{d-1}} g(x,-x)\,d\lambda_{d-1}(x),
}
for any continuous function~$g\in C(S^{d-1}\times S^{d-1})$. Then,
\eq{
\int\limits_{S^{d-1}\times S^{d-1}}\!\!\!\!\! \rho(x,y)\,d\mu(x,y) = \int\limits_{S^{d-1}} \rho(x,-x)\,d\lambda_{d-1}(x) = 2,
}
whereas
\eq{
\int\limits_{S^{d-1}\times S^{d-1}}\!\!\!\!\! |f(x) - f(y)|\,d\mu(x,y) = \int\limits_{S^{d-1}} |f(x) - f(-x)|\,d\lambda_{d-1}(x) \leq 12 e^{-d\eps^2/8} + 2\eps
}
for any~$\eps$ by~\eqref{ConcentrationOfMeasure}. The choice~$\eps =  4d^{-\frac12} \sqrt{\log d}$ completes the proof.
\end{proof}

\section{Martingale method for bounding~$\LCJ$}\label{S4}
We start with a technical proposition that suggests a unified approach to Theorems~\ref{LaaksoTheorem} and~\ref{TreeTheorem}. 
\begin{Prop}\label{TechnicalMartingaleProposition}
Fix~$N\in \mathbb{\N}$ and assume a metric space~$\Metr$ fulfills two requirements.
\begin{enumerate}[1.]
\item There exists a finite filtration~$\FF = \{\FF_n\}_{n=1}^N$ such that for each~$n =1,2,\ldots, N$ the atoms of~$n$-th level have equal probability and for any~$\omega \in \AF_n$ there is a point~$x_\omega\in \Metr$ assigned to it. Furthermore\textup, the atoms of~$\AF_n$ are split into pairs\textup; if~$\omega_1$ and~$\omega_2$ are paired\textup, the constant~$c_{\omega_1,\omega_2}$ is defined by the formula
\eq{
P(\omega_1) = P(\omega_2) = c_{\omega_1,\omega_2} \rho(x_{\omega_1},x_{\omega_2}).
}
\item For any~$f\in \LIP_1(\Metr)$ there exists a martingale~$M$ adapted to~$\FF$ such that~$\|M\|_{L_\infty} \leq 1$ and for any~$n=1,2,\ldots, N-1$ and any pair of atoms~$\omega_1,\omega_2\in \AF_n$\textup, 
\eq{\label{eq412}
\E|d M_{n+1}|\chi_{\omega_1 \cup \omega_2} \gtrsim c_{\omega_1,\omega_2} |f(x_{\omega_1}) - f(x_{\omega_2})|.
}
\end{enumerate}
Then\textup,~$\ELL(\Metr)\lesssim N^{-\frac12}$.
\end{Prop}
\begin{proof}
We wish to use Corollary~\ref{MeasureRemark} and construct the measure
\eq{
\mu = \sum\limits_{n=1}^{N-1} \sum\limits_{\AF_n} c_{\omega_1,\omega_2} \delta_{(x_{\omega_1},x_{\omega_2})}; 
}
on~$\Metr\times \Metr$; in the interior summation we mean summation over all pairs of atoms in~$\AF_n$.  One the one hand,
\eq{
\int\limits_{\Metr\times \Metr} \rho(x,y)\,d\mu(x,y)= \sum\limits_{n=1}^{N-1}\sum\limits_{\AF_n} c_{\omega_1,\omega_2}\rho(x_{\omega_1},x_{\omega_2}) = \sum\limits_{n=1}^{N-1}\sum\limits_{\AF_n} P(\omega_1) = \frac{N-1}{2}.
}
On the other hand, for any~$f\in \LIP_1(\Metr)$, we construct the martingale~$M$ as described above and get
\mlt{
\int\limits_{\Metr\times \Metr} |f(x) - f(y)|\,d\mu(x,y) = \sum\limits_{n=1}^{N-1}\sum\limits_{\AF_n} c_{\omega_1,\omega_2}\big|f(x_{\omega_1}) - f(x_{\omega_2})\big|\\
\lesssim \sum\limits_{n=1}^{N-1}\sum\limits_{\AF_n} \E|d M_{n+1}|\chi_{\omega_1 \cup \omega_2} = \sum\limits_{n=1}^{N-1}\E |dM_{n+1}| \Leqref{eqq.key_martingal_ineq} \sqrt{N}.
}
Thus, Corollary~\ref{MeasureRemark} yields~$\ELL(\Metr)\lesssim N^{-\frac12}$.
\end{proof}
\begin{proof}[Proof of~\eqref{WholeTreeInequality}.]
Without loss of generality, we may assume the depth of the tree is a power of two, and, thus replace~$N$ with~$2^N$ in our notation. We wish to apply Proposition~\ref{TechnicalMartingaleProposition} and start with a convenient description of the filtration~$\FF$. Let the square~$[0,1]^2$ equipped with the standard Lebesgue measure be our probability space~$\Omega$. We enumerate the leaves of~$\TT$ with the numbers~$1,2,\ldots, 2^{2^N}$ and to each of them assign the rectangle
\eq{
P_j = \Big[\frac{j-1}{2^{2^N}},\frac{j}{2^{2^N}}\Big]\times [0,1]\subset \Omega
}
in a natural way. 

The~$\sigma$-field~$\FF_0$ is generated by the rectangles~$P_1, P_2,\ldots, P_{2^{2^N}}$. The~$\sigma$-field~$\FF_k$ is generated by the rectangles
\eq{
\Big[\frac{j-1}{2^{2^N}},\frac{j}{2^{2^N}}\Big]\times \Big[\frac{\ell-1}{2^k},\frac{\ell}{2^k}\Big],\qquad \ell = 1,2,\ldots, 2^k,\quad j=1,2,\ldots,2^{2^N}.
}
\begin{figure}[h!]
\centerline{
\includegraphics[height=5cm]{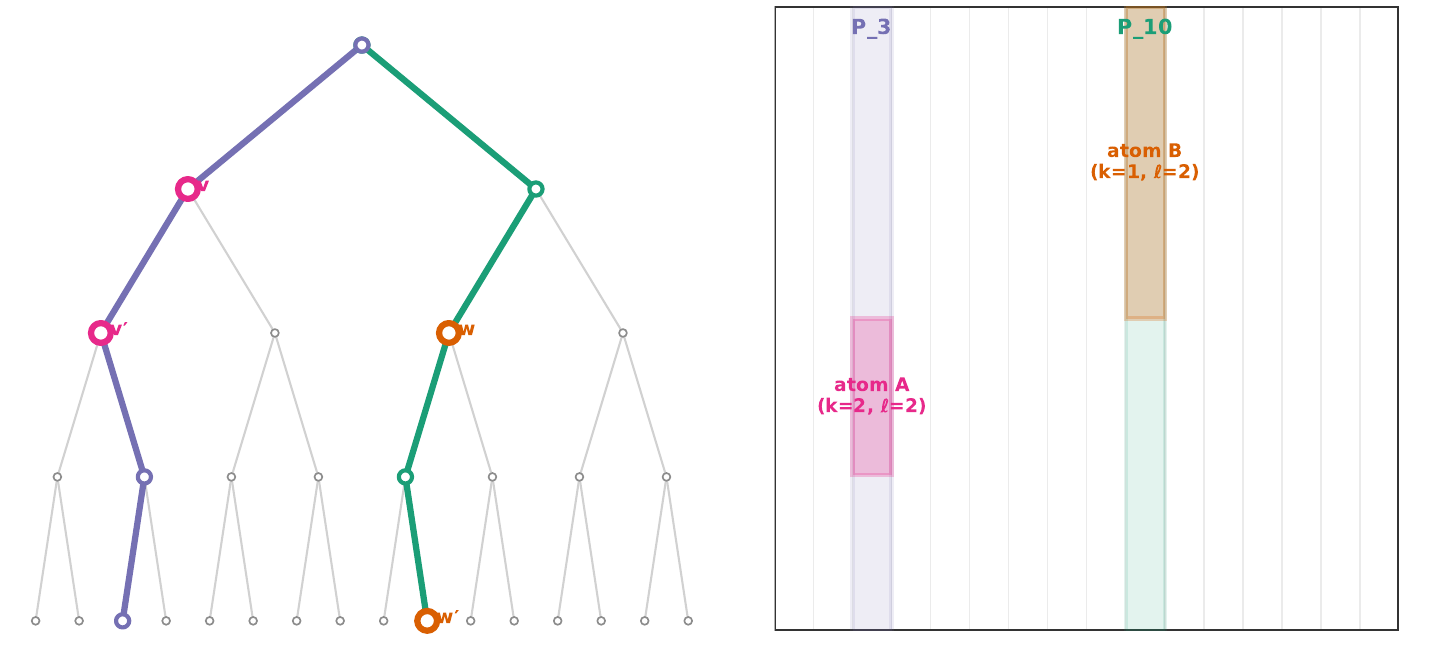}
}
\caption{Two vertices and corresponding atoms for~$N=2$.}
\label{Fig:TreeMartingale}
\end{figure}

Note that the probability of an atom depends on its level~$k$ only as required by Proposition~\ref{TechnicalMartingaleProposition}. We will identify atoms of our filtration with the triples~$(k,j,\ell)$ in a natural way. To each atom of~$\FF_k$, we assign two vertices in the tree: we consider the path from the root to the~$j$-th leaf and denote the~$(\ell-1)2^{N-k}$-th and~$\ell 2^{N-k}$-th vertices on this path, by~$v_{k,j,\ell}$ and~$v_{k,j,\ell}'$, respectively. Note that
\eq{
\rho(v_{k,j,\ell}, v_{k,j,\ell}') = 2^{N-k}.
}
See Fig.~\ref{Fig:TreeMartingale} for visualization of an example. As we shall see, the role of the vertice~$v_{k,j,\ell}$ is auxiliary, and~$v'_{k,j,\ell}$ will serve as~$x_{\omega}$ in the terminology of Proposition~\ref{TechnicalMartingaleProposition} for~$\omega \in \FF_k$ corresponding to~$(k,j,\ell)$. So far we have described the filtration~$\FF$ and the correspondence~$\omega \longleftrightarrow x_\omega$. 

Next, we describe how to construct the martingale~$M$ adapted to~$\FF$ from~$f\in \LIP_1(\TT)$. We construct~$M$ in the following way: On the atom~$(k,j,\ell)$ we set
\eq{
M_k(k,j,\ell) = 2^{k-N}\Big(f(v_{k,j,\ell}) - f(v_{k,j,\ell}')\Big).
}
One may see this is a martingale indeed, and the Lipschitz bound for~$f$ yields~$\|M\|_{\infty} \leq 1$; both assertions are completely analogous to Example~\ref{OneDimensionalExample}. In fact, the martingale constructed here is nothing but a concatenation of~$2^{2^N}$ martingales of the type described in Example~\ref{OneDimensionalExample} generated by the restrictions of~$f$ to paths from the root of~$\TT$ to a leaf.

Now we describe the martingale differences. Pick an atom~$(k,j,\ell)$ and let~$u_{k,j,l}$ be the midpoint of the segment~$[v_{k,j,\ell},v_{k,j,\ell}']$; we assume~$k < N$. Then, according to~\eqref{eq210},
\eq{
\Big|d M_{k+1}(k,j,\ell)\Big| = 2^{k-N} \Big |f(v_{k,j,\ell}) + f(v_{k,j,\ell}') - 2f(u_{k,j,\ell})\Big|.
}

It is high time to split our atoms into pairs. Consider all the atoms~$(k,j,l)$ for which~$u_{k,j,\ell} = U$, some fixed vertex in the tree. One may see that the choice of~$u_{k,j,\ell}$ fixes the parameters~$k$ and~$\ell$, while~$j$ may still vary; the vertex~$v_{k,j,\ell}$ is fixed and there are several choices possible for~$v_{k,j,\ell}'$. We note that the vertices considered are split into two groups of equal sizes: the ones for which the path from~$W$ downto~$v_{k,j,\ell}'$ turns left at~$U$ and the ones for which it turns right. We split the vertices into pairs with 'left' and 'right' vertices in each pair, in an arbitrary manner. Let~$v_1'$ and~$v_2'$ be two vertices in such a pair corresponding to the indices~$(k,j_1,\ell)$ and~$(k,j_2,\ell)$ and to atoms~$\omega_1$,~$\omega_2$ in~$\AF_n$. Then,
\mlt{\label{BoundindMartingaleDifference}
\E |d M_{k+1}| \chi_{\omega_1\cup \omega_2}\\
 = 2^{-2^N - k} \cdot 2^{k-N} \Big(\big|f(v) + f(v_1') - 2f(U)\big| + \big|f(v) + f(v_2') - 2f(U)\big|\Big)\\
  \geq 2^{-2^N - N}\big|f(v_1') - f(v_2')\big|.
}
Note that~$P((k,j,\ell)) = 2^{-2^N - k}$ and~$\rho(v_1',v_2') = 2^{k-N+1}$, which implies~\eqref{eq412} with~$c_{\omega_1,\omega_2} = 2^{-2^{N} - N}$. The application of Proposition~\ref{TechnicalMartingaleProposition} concludes the proof.
\end{proof}

\begin{proof}[Proof of Theorem~\ref{LaaksoTheorem}]
We will employ the ideas of the previous proof. Since the details are also quite similar, we omit some of them. Recall the notation introduced in Subsection~\ref{sLaakso}. It suffices to prove the bound
\eq{
\LCJ(X_N) \lesssim N^{-\frac12}.
}
To do this, we will employ Proposition~\ref{TechnicalMartingaleProposition}. 

We start with the description of the filtration~$\FF$. Consider all the paths from~$0$ to~$1$ in~$G_N$ that go from left to the right. Let their total number be~$M$. We split the square~$[0,1]^2$ into~$M$ rectangles of dimensions~$[0,1]\times M^{-1}$ in a natural way and match a path to such a rectangle. This defines the atoms of~$\FF_0$; denote the rectangle matching a path~$\gamma$ by~$R_\gamma$. For each~$k=1,2,\ldots, N$, we consider the restriction of~$\gamma$ to~$X_k$ and then extend it to~$G_k$, where it becomes a path of~$4^k$ segments of length~$4^{-k}$. This naturally defines the splitting of~$R_\gamma$ into the rectangles of the type
\eq{
\Big[\frac{j-1}{4^k}, \frac{j}{4^k}\Big]\times \Big[\frac{s-1}{M}, \frac{s}{M}\Big],\qquad j=1,2,\ldots, 4^k, \quad R_\gamma = [0,1]\times \Big[\frac{s-1}{M}, \frac{s}{M}\Big].
}
Such type rectangles define the atoms of~$\FF_k$. Note that atoms of~$\FF_k$ have the same probabilities.

We have constructed the filtration and now turn to construction of martingales from Lipschitz functions. For each path~$\gamma$ of the type described above and~$f\in \LIP_1(X_N)$, we consider the function~$f|_{\gamma}$ as a function on an interval (extend it to piecewise affine function). This interval is equipped with the natural four-adic filtration inherited from~$\FF$. The filtration, in its turn, generates a bounded martingale similar to the one in Example~\ref{OneDimensionalExample}.

We need to assign a vertex in~$X_N$ to each atom in our filtration. By the construction, to each atom~$\omega$ of~$k$-th level there corresponds a path~$\gamma$ in~$G_N$ and an edge~$\ell \in G_k$ connecting two neighbor vertices in~$G_k$, both lying on~$\gamma$. In the graph~$G_{k+1}$,~$\ell$ is replaced with four edges whose endpoints lie on~$\gamma$, and we choose the common endpoint of the two middle edges for the role of~$x_\omega$ in the terminology of Proposition~\ref{TechnicalMartingaleProposition}.

Finally, we need to describe the pairing of atoms in~$\FF_k$. We fix an edge~$\ell \in G_k$ and consider all the atoms of the filtration that correspond to it (whose paths pass through both endpoints of~$\ell$). They naturally split into two groups according to the chosen point (the ones that choose the `upper' midpoint and the ones that choose the `lower' one), and these groups have equal cardinalities. Thus, we may pair the `upper' atoms with the `lower' ones in an arbitrary manner.
\begin{figure}[h!]
\centerline{
\includegraphics[height=3cm]{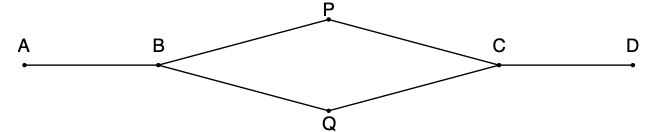}
}
\caption{Illustration to the proof of the analog of~\eqref{BoundindMartingaleDifference}}
\label{Fig:SixPoints}
\end{figure}

It remains to verify~\eqref{eq412}. The computation is similar to~\eqref{BoundindMartingaleDifference}.
Consider two paths that go through~$\ell = [A,D]$ in~$G_k$ and such that they pass through different points~$P, Q$ in~$G_{k+1}$ (see Fig.~\ref{Fig:SixPoints} for visualization). We have~$\rho(P,Q) = 2\cdot 4^{-k-1}$ and~$P(\omega_1) = P(\omega_2) = 4^{-k}M^{-1}$. Therefore,~$c_{\omega_1,\omega_2} = 2 M^{-1}$. Then, the corresponding part of the mathematical expectation is
\mlt{
M^{-1}4^{-k-1}\cdot 4^k \Big(2|f(D) + 3 f(A) - 4 f(B)| + 2|f(A) + 3 f(D) - 4 f(C)| + |f(D) - f(A) + 4 f(B) - 4 f(P)| \\
+ |f(D) - f(A) - 4 f(C) + 4 f(P)| + |f(D) - f(A) + 4 f(B) - 4 f(Q)| + |f(D) - f(A) - 4 f(C) + 4 f(Q)| \Big) \\
  \geq \frac{1}{4M}|f(P) - f(Q)|,
}
which yields~\eqref{eq412} for our graph. The application of Proposition~\ref{TechnicalMartingaleProposition} concludes the proof.
\end{proof}

\section{Proof of Theorem~\ref{UltrametricTheorem}}\label{S5}
Without loss of generality, we assume~$\Metr$ is separable and has diameter~$1$. Note that bi-Lipschitz change of metric does not affect the property of belonging to the~$\LCJs$-class. Thus, by Proposition~\ref{Prop.ultrametric}, we may assume that $\rho$ is an ultrametric on $\Metr$.

We will construct a random~$1$-Lipschitz function~$\Psi$ on~$\Metr$ such that
\eq{\label{Expectation}
\E \big|\Psi(x) - \Psi(y)\big| \gtrsim \rho(x,y),\qquad \text{for any\ } x,y\in \Metr.
}
Clearly, this proves the theorem. To be more precise, we will construct a random~$C$-Lipschitz function~$\Phi$,~$C\in \R$, such that it also satisfies~\eqref{Expectation}; then, we may set~$\Psi = \Phi/C$.

For any given~$r$, the ultrametric space~$\Metr$ may be split into the balls~$B_{r}(x_j)$ in such a way that
\eq{
\rho(B_r(x_j), B_r(x_i)) \geq r, \qquad i\ne j,
}
see Section 15.3 in \cite{DavidSemmes1997} for details. Pick~$q\in (0,1)$ to be sufficiently small number. Fix the splittings described above for~$r = q^n$,~$n\in \N$, and denote the centers of the balls in the~$n$-th splitting by~$x_{n,j}$. Let~$\{\eps_{n,j}\}_{n,j}$ be a family of the independent Bernoulli random variables. We are ready to define the desired function~$\Phi$:
\eq{
\Phi(x) = \sum\limits_{n=1}^\infty q^{n}\sum\limits_j \eps_{n,j}\chi_{B_{q^{-n}}(x_{n,j})}(x),\qquad x\in \Metr.
}
We need to prove two claims:~$\Phi$ is almost surely~$C$-Lipschitz for sufficiently large~$C$, and~\eqref{Expectation}. Let us start with the first claim. Fix~$x$ and~$y$ in~$\Metr$ and let~$n_0$ be the smallest integer~$n$ such that~$x$ and~$y$ belong to different balls of the~$n$-th partition. In particular,~$\rho(x,y) \asymp q^{n_0}$. Then, by the construction,
\eq{
|\Phi(x) - \Phi(y)|  = \Big|\sum\limits_{k \geq n_0} q^{n}(\xi_k -\eta_k)\Big| \leq \frac{2q^{n_0}}{1-q} \lesssim \rho(x,y);
}
where~$\xi_k$ and~$\eta_k$ are some of the~$\eps_{n,j}$. Here we have used the property that the family~$\{B_{q^{-n}}(x_{n,j})\}_j$ is disjoint. Thus, the constructed function is~$C$-Lipschitz indeed, where~$C = C(q)$.

Let us prove~\eqref{Expectation}. By the same formula,
\eq{
\E \big|\Phi(x) - \Phi(y)\big|  \geq q^{n_0}\E |\xi_{j_0} - \eta_{j_0}| - \sum\limits_{k > n_0} q^{n}\E\big|\xi_k -\eta_k\big| \geq \frac12 q^{n_0}- \frac{q^{n_0+1}}{1-q}.
}
This is bounded away from zero by~$\frac14  q^{n_0}\asymp \rho(x,y)$, provided~$q$ is sufficiently small. We fix such~$q$ and obtain~\eqref{Expectation} and the uniform Lipschitz bound for~$\Phi$.

\bibliography{/Users/mac/Documents/Bib/Mybib}{}
\bibliographystyle{plain}

\end{document}